\documentclass[a4paper,11pt]{amsart}

\usepackage{amsmath}
\usepackage{amssymb,amsbsy,amsmath,amsfonts,amssymb,amscd}
\usepackage{latexsym}
\usepackage{amsthm}
\usepackage{graphics}
\usepackage{color}
\usepackage{tikz}
\usetikzlibrary{arrows}
\usepackage{tikz-cd}
\usepackage{pdfpages} 
\usepackage{longtable}
\usepackage{xy}
\usepackage{array}
\usepackage{color}
\usepackage{comment}

\input xy
\xyoption{all}

\newcommand\Oh{{\mathcal O}}
\newcommand\sC{{\mathcal C}}
\newcommand\sT{{\mathcal T}}
\newcommand\sD{{\mathcal D}}
\newcommand\sE{{\mathcal E}}
\newcommand\sA{{\mathcal A}}

\newcommand\sX{{\mathcal X}}
\newcommand\sY{{\mathcal Y}}
\newcommand\sH{{\mathcal H}}

\newcommand\Lam{\Lambda}
\newcommand\al{\alpha}

\newcommand\s{\sigma}

\newcommand\Ga{\Gamma}

\newcommand\fie{\varphi}
\newcommand\Om{\Omega}

\DeclareMathOperator{\Hom}{Hom}

\DeclareMathOperator{\Gal}{Gal}

\DeclareMathOperator{\Def}{Def}

\newcommand{\CC}{\ensuremath{\mathbb{C}}}
\newcommand{\RR}{\ensuremath{\mathbb{R}}}
\newcommand{\ZZ}{\ensuremath{\mathbb{Z}}}
\newcommand{\QQ}{\ensuremath{\mathbb{Q}}}

\newcommand{\NN}{\ensuremath{\mathbb{N}}}
\newcommand{\hol}{\ensuremath{\mathcal{O}}}

\newcommand{\ra}{\ensuremath{\rightarrow}}

\newcommand{\frX}{\ensuremath{\mathfrak{X}}}

\def\eea{\end{eqnarray*}}
\def\bea{\begin{eqnarray*}}

\DeclareMathOperator{\Aut}{Aut}
\DeclareMathOperator{\End}{End}

\DeclareMathOperator{\Sing}{Sing}

\newcommand{\Proof}{{\it Proof. }}

\newcommand\dual{\mathrel{\raise3pt\hbox{$\underline{\mathrm{\thinspace d
\thinspace}}$}}}
\newcommand\qe{\ifhmode\unskip\nobreak\fi\quad $\Box$}       

\def\BOX{\hfill\lower.5\baselineskip\hbox{$\Box$}}

\newtheorem{theorem}{Theorem}
\newtheorem{theo}[theorem]{Theorem}
\newtheorem{remark}[theorem]{Remark}
\newenvironment{rem}{\begin{remark}\rm}{\end{remark}}

\newtheorem{prop}[theorem]{Proposition}

\newtheorem{lemma}[theorem]{Lemma}
\newtheorem{example}[theorem]{Example}
\newenvironment{ex}{\begin{example}\rm}{\end{example}}

\theoremstyle{definition}
\newtheorem{defin}[theorem]{Definition}

\newcommand{\torus}{\ensuremath{T}}

\DeclareMathOperator{\Irr}{Irr}

\DeclareMathOperator{\Bihol}{Bihol}

\DeclareMathOperator{\id}{id}

\usepackage{hyperref}

\makeatletter
\def\tagform@#1{\maketag@@@{\ignorespaces#1\unskip\@@italiccorr}}
\makeatother

\newcolumntype{H}{@{}>{\lrbox0}l<{\endlrbox}} 

\begin{document}

\title[Rigid Group Actions on Complex Tori are Projective]{Rigid Group Actions on Complex Tori are Projective (after Ekedahl)}
\author{Fabrizio Catanese and  Andreas Demleitner}
\address {Lehrstuhl Mathematik VIII\\
Mathematisches Institut der Universit\"at Bayreuth, NW II\\
Universit\"atsstr. 30\\
D-95447 Bayreuth \\
Germany}
\email{Fabrizio.Catanese@uni-bayreuth.de \newline \hspace*{2.35cm} Andreas.Demleitner@uni-bayreuth.de}

\thanks{
The authors acknowledge support of the ERC 2013 Advanced Research Grant - 340258 - TADMICAMT}

\begin{abstract}
In this note, we give a detailed proof of a result due to Torsten Ekedahl,  
 describing complex tori admitting a rigid  group action and showing explicitly  their projectivity  and their structure
 in terms of CM-fields.
In the appendix, joint with Claudon, we show, using a method of Green-Voisin,  that all group actions on complex tori deform to projective ones.
\end{abstract}
\maketitle

\setcounter{section}{0}

\section{Introduction}

 The work of Kodaira \cite{kod}  \cite{kod2}  lead to the question whether any compact K\"ahler manifold   enjoys the property of admitting arbitrarily small deformations which are projective (Kodaira settled in  \cite{kod2}  the case of surfaces).
 
 Motivated by Kodaira's problem (see the final  section and the appendix) the first author  asked Torsten Ekedahl at an Oberwolfach conference
   around  1999    if there exists a rigid group action of a finite group $G \subset \Bihol(T)$ on a complex torus $T$ (see section  \ref{prelim} for  definitions regarding deformations of group actions) which is not projective. T. Ekedahl answered this question and sketched a strategy of proof for the statement that the rigidity of the action $(T,G)$ implies that $T$ is projective (i.e., 
 $T$ is an abelian variety).

  Later Claire Voisin  gave a counterexample to the general Kodaira problem showing in   \cite{Voisin} the existence of a rigid compact K\"ahler manifold which is not projective (and later in \cite{Voisin2}  she even gave counterexamples which are not bimeromorphic to a projective manifold).  Kodaira's property  still remains a very interesting  theme of research:  understanding  which compact K\"ahler manifolds or  K\"ahler spaces with klt singularities satisfy Kodaira's property
 (see   \cite{graf} for quite recent progress). 
  
  On the other hand Ekedahl's    approach allows a rather explicit description of rigid actions on complex tori in terms of orders in CM-fields, hence
  providing   explicitly given polarizations on them. Therefore  his result turned out  to be quite interesting and useful for other purposes (see \cite{Demleitner} for applications to the classification theory of quotient manifolds of complex tori), and for this reason we find it important to publish 
  here a complete proof.

\begin{theorem}[Ekedahl] \label{ekedahl}
	Let $(\torus,G)$ be a rigid group action of a finite group $G \subset \Bihol(T)$ on a complex torus $\torus$. Then $\torus$ (or, equivalently, $T/G$) is projective.    Moreover, if we write $T = V^{1,0} / \Lambda$, then $$  \Lambda \otimes_{\ZZ} \QQ = \oplus _j W_j^{n_j}, $$
	where $W_j$ is a Hodge structure on a CM field $F_j$ and where $\oplus_j F_j$ is a subalgebra of the centre of the group algebra $\QQ[G]$.
\end{theorem}

The contents  of the paper are as follows. 
	
	 In section \ref{prelim}, we briefly discuss deformations of group actions on complex manifolds. 

	 Then, in the subsequent section 
		\ref{Prel2},   we develop the tools used in  the proof of Theorem \ref{ekedahl}, mainly  based on Hodge theory and representation theory.
		
		The main ideas of the proof are the following:  if $\sA$ is a finite-dimensional semisimple $\QQ$-algebra, the rigidity of the action of $\sA$ (cf. Definition \ref{def-rigid}) on a rational Hodge structure $V$ of weight $1$ can be determined  by looking at the simple summands of  $\sA \otimes_{\QQ} \CC$	appearing in $V^{1,0}$, respectively in $V^{0,1}$. A second ingredient is that, for  $\sA = \QQ[G]$ with $G$ finite (and also in a more general situation), we show that rigidity is equivalent to  having a rigid action of the commutative subalgebra given by the centre $Z(\QQ[G])$.
		 
		 Then we apply 
		Proposition \ref{polarizable}, stating  that, if $\sA = Z(\QQ[G])$ is the centre of the group algebra  and the action of $\sA$ on $V$ is rigid,	then the  Hodge structure $V$ is polarizable.
		
		 Finally, in the appendix, we  show that every group action $(T,G)$ on a complex torus admits arbitrarily small deformations which are projective.

\section{Deformations of group actions} \label{prelim}

Let $X$ be a compact complex manifold. Let  $G \subset \Bihol(X)$ be a finite group, and denote by $\al \colon G \times X \to X$ the corresponding group action of $G$ on $X$.

\begin{defin}
1) A \textit{deformation $(p, \al')$} of the group action $\al$ of $G$ on $X$ consists  of a deformation 
 $p \colon (\frX ,\frX_0)  \to (B,t_0)$ of $X$ (i.e., $\frX_0 : = p^{-1} (t_0)$ and $X \cong \frX_0$)  given together with  
 $\al' \colon G \times \frX \to \frX$, a holomorphic group action commuting with $p$ (here we let $G$ act trivially on the base), such that the action on $\frX_0 \cong  X$ induces  the initially given action $\al$.

2) A deformation $(p, \al')$ is said to be  {\em trivial} if  its germ is isomorphic to the trivial deformation
$ X \times B \ra B$, endowed with the action $\al \times \id_B$.

3) The action $\al$ is said to be {\em rigid} if every deformation of $\al$ is trivial.
\end{defin} 

Kuranishi theory leads to an easy characterization  of rigidity  of an  action $\al$ of a group $G$ on $X$,
see \cite[p. 23]{Catanese-CIME}, \cite[Ch. 4]{Catanese-Guide}, \cite{Binru}. 

Denote by $p \colon \frX \to \Def(X)$ the Kuranishi family of $X$; then  this characterization is related to the  question: which condition  on $t \in \Def(X)$ guarantees that $G$ is a subgroup of $\Aut(\frX_t)$? It turns out (cf. \cite[p. 23]{Catanese-CIME}) that $G \subset \Bihol(\frX_t)$ if and only if $g_* t = t$ for any $g \in G$, so that $t \in \Def(X) \cap H^1(X, \theta_X)^G$. 

We then have (see proposition 4.5 of \cite{Catanese-Guide}):

\begin{prop}
Set	$\Def(X)^G := \Def(X) \cap H^1(X, \theta_X)^G$. The group action $\al$ of $G$ on $X$ is  rigid if 
and only if $\Def(X)^G = 0$ (as a set). A fortiori the action is rigid if $H^1(X, \theta_X)^G = 0$  (in this latter case we say that the action is infinitesimally rigid).
\end{prop}

In the upcoming chapter we shall consider the case where  $X = T$ is a complex torus: the rigidity of $(T,G)$, 
amounting to the fact that the representation of $G$ on $H^1(X, \Theta_X)$ contains no trivial summand,
can then be read off explicitly from the action of $G$ on the tangent bundle.

\section{ Rigid actions on rational Hodge structures}\label{Prel2}

 Denote by  $\sH^1$ the category of rational Hodge structures of type $((1,0),(0,1))$. An object of $\sH^1$
  is a finite-dimensional $\QQ$-vector space $V$  endowed with a decomposition
\[
V \otimes_{\QQ} \CC  = U \oplus \overline{U}  = : V^{1,0} \oplus V^{0,1}.
\]

The elements of $\sH^1$ can be viewed as isogeny classes of complex tori $$T: = (\Lam \otimes_{\ZZ} \CC )/  (\Lam \oplus V^ {0,1})  ,$$
where $\Lam \subset V$ is an {\em order}, i.e. a  free subgroup of maximal rank (by abuse of notation we shall also say that $\Lam$ is a lattice in $V$,
observe that $V = \Lam \otimes_{\ZZ} \QQ$).

We have isogeny classes of Abelian varieties when a rational Hodge structure is polarizable, according to the following

\begin{defin}
	Let $V \in \sH^1$ and write for short $V_\CC := V \otimes_{\QQ} \CC$. 
	
	A \textit{polarization on $V$} is an alternating form  $E \colon V \times V \to \QQ$ satisfying the two Hodge-Riemann Bilinear Relations: 
	
	i)  The complexification $E_\CC \colon V_\CC \times V_\CC \to \CC$ satisfies $E_\CC(V^{1,0},V^{1,0}) = 0$ (hence also $E_\CC(V^{0,1},V^{0,1}) = 0$)
	
	 ii) For any non-zero vector $v \in V^{1,0}$, we have $  -  \ i \cdot E_\CC(v,\overline{v}) > 0$ 
	 	 
	 Equivalently, setting  $E_\RR \colon V_\RR \times V_\RR \to \RR$, we have:
	 
	 I) $E_\RR(J  x, J y) = E_\RR(  x,  y) $
	 
	 II) the symmetric bilinear form $E_\RR  (x, J  y)$ is positive definite.
	 
	 Here, if $ x = u + \bar{u}$, $Jx : = iu - i \bar{u}$ ($J^2 = - Id$).
\end{defin}

Let $\sA$ be a semisimple and finite-dimensional $\QQ$-algebra (for example the group algebra $\sA = \QQ[G]$ for a finite group $G$). We denote an action $r \colon \sA \to \End_{\sH^1}(V)$ for $V \in \sH^1$ by a triple $(V,\sA,r)$.

If  $\Lam \subset V$ is a lattice and $\torus = (V \otimes_\QQ \CC)/  (\Lam \oplus V^ {0,1}) $ is the corresponding complex torus then  $\sA$ maps  to $\End (T)\otimes _{\ZZ}\QQ $.

\begin{defin} \label{def-rigid}
An action $(V,\sA,r)$ is called \textit{rigid}, if 
\begin{align} \label{rigidity}
 \Hom_{\sA} (V^{0,1}, V^{1,0}) = 0.
\end{align}
\end{defin}

Rigidity  \ref{rigidity} means, in view of what we saw in the previous section, and  in view of 
\[
H^1(\Theta_{\torus}) =  H^1(\Oh_{\torus}) \otimes_\CC H^0(\Om^1_{\torus})^\vee =  \overline{U}^\vee \otimes_\CC U  =  \Hom_{\sA} (V^{0,1}, V^{1,0}),
\]
that there are no deformations of $\torus$ preserving the $\sA$-action.\\

We consider now some examples of the above notion.

\begin{ex}\label{imaginary} 
Let $ \sA$ be a totally imaginary number field $F$. This means that  $[F : \QQ] = 2k$ and $F$ possesses $2k$ different embeddings 
$\s_j : F \ra \CC$, none of which  is  \textit{real} (this means: $\s_j (F) \subset \RR$). 

Hence each $\s_j$ is different from the complex conjugate, $\sigma_j \neq  \overline{\sigma_j}$,
and if we set $V : = F$, with the obvious action of $F$, all the Hodge structures on $V$ are rigid
and correspond to the finite set of partitions of the set $\sE$ of embeddings of $F$  into
two  conjugate sets $\{ \s_1, \dots, \s_k\} $ and  $\{ \overline{\s_1}, \dots, \overline{\s_k} \} $.

Since  the $F$-module $F \otimes_{\QQ} \CC$ is the direct sum 
$$ F \otimes_{\QQ} \CC = \oplus_{\s_j \in \sE} \  \CC_{\s_j} ,$$
where  $\CC_{\s_j} $ is the vector space $\CC$ with left action of $F$ given by:
$$ x \cdot z : = \s_j(x) \cdot z, \ \forall x \in F , z \in \CC, $$
 and choosing  such a  partition amounts to choosing  $ V^{1,0} : = \oplus_{j=1, \dots k} \CC_{\s_j} $.
\end{ex}

A particular case is given by the class of CM-fields.

\begin{ex}\label{CM} 

Recall  that a \textit{CM-field} is a totally imaginary quadratic extension $F$ of a totally real number field $K$. 

Equivalently, (cf. \cite[Proposition 5.11]{Shimura}) $F$ is a CM-field  if it carries a non-trivial involution $\rho$ such that $\sigma \circ \rho = \overline{\sigma}$ for all embeddings $\sigma \colon F \hookrightarrow \CC$ .
In particular $F$ is totally imaginary.

In this case any Hodge structure on $V : = F$ is polarizable.

			 Let indeed $\sigma_1,...,\sigma_k \colon F \hookrightarrow \CC$ be the embeddings of $F$ occurring in $V^{1,0}$. Following  \cite[p. 128]{Shimura}  choose $\zeta \in F$ satisfying the following conditions:
			\begin{itemize}
				\item[a)] $\zeta$ is imaginary, i.e., $\rho(\zeta) = - \zeta$,
				\item[b)] $\sigma_j(\zeta)$ is imaginary with  positive imaginary part for each  $j = 1,..., k$.  
			\end{itemize}
			A polarization on $V$ of $F$ is then given,  if we set $x_j := \sigma_j(x),  y_j := \sigma_j(y)$, by
			the skew symmetric form  (we set here $\s_{k+j} : = \overline{ \s_j}$)
			\begin{align*}
			E(x,y) := tr_{F/\QQ}(\zeta x \rho (y)) = 
			 \sum_{j=1}^{2k}  \sigma_j(\zeta)x_j \overline{y_j}  = \sum_{j=1}^k \sigma_j(\zeta)(x_j \overline{y_j} - \overline{x_j} y_j). 
			\end{align*}
			In fact,   the first Riemann bilinear relation amounts to $E(J x, J y) = E (x,y)$, which is clearly satisfied, 
			since $(J x)_j = i x_j ,$ for $ j = 1, \dots, k$, and the  real part of the associated  Hermitian form 
			is the  symmetric  form 
			 $$E(x, J   y) = \sum_{j=1}^k  (-i) \sigma_j(\zeta) ( x_j \overline{y_j} + \overline{x_j} y_j),$$ 
			which is positive definite since 
			$$E(x, J  x) = \sum_{j=1}^k  2 \  Im (\sigma_j(\zeta))   |x_j|^2 > 0 $$
			 for $x \neq 0$. \\

\end{ex}

 \bigskip

Let us  now proceed towards the proof of the main theorem.

An important  step towards the main  Theorem is that in the case where 
\begin{align}\sA = \QQ[G] \end{align}
rigidity can be reduced to rigidity of the action restricted to 
the centre of the group algebra.

\begin{prop} \label{reduce_to_comm}
	Let $\sA = \QQ[G]$ be the group algebra of a finite group $G$ over the rationals. 
	
	Then the  triple $(V,\sA,r)$ is rigid if and only if $(V, Z(\sA),r')$ is rigid, where $Z(\sA)$ is the centre of $\sA$ and $r'$ is the restriction of $r$ to $Z(\sA)$. 
\end{prop}

\begin{proof}

For each field $K$, $\QQ \subset K \subset \CC$, 
$\sA \otimes_\QQ K = K[G] $ has as  centre $Z_{K} : = Z (K[G])$,  the  vector space 
with basis $ v_{\sC}$, indexed by the conjugacy classes $\sC$ of $G$, and where
$$ v_{\sC} : = \sum_{g \in \sC} g.$$

For $K = \CC$, another more useful basis is  indexed by the irreducible complex representations $W_{\chi}$ of $G$,
and their characters $\chi$ (these form an orthonormal basis for the space of class functions,
i.e. the space $Z_{\CC}$ if we identify the element $g$ to its characteristic function).

	For each irreducible $\chi$, the element
	\[
	e_\chi := \frac{\chi(1)}{|G|} \sum_{g \in G} \chi(g^{-1}) \cdot g \in \CC[G]
	\]
	is an  idempotent in $Z(\CC[G])$. Indeed, we even have that
	\[
	Z(\CC[G]) = \bigoplus_\chi \CC \cdot e_\chi,
	\]
	and the  idempotents $e_\chi$  satisfy the orthogonality relations $e_{\chi'} \cdot e_{\chi} = 0$ for $\chi \neq \chi'$.
	
	This leads directly to the decomposition 
	\[
    \sA \otimes_\QQ \CC =  \CC[G] = \bigoplus_{\chi \in Irr} A_\chi, \; \; A_\chi : =    e_\chi   \CC[G]  \cong \End(W_\chi),
	\]
	where $\chi$ runs over all irreducible characters of $G$, and to the semisimplicity of the group algebra. 
Notice that $e_\chi$ acts as the identity  on $W_\chi$, and as $0$ on $W_{\chi'}$ for $\chi' \neq \chi$. 

	In fact, we shall prove the stronger statement that for any two finitely generated $\CC[G]$-modules $M$ and $N$ (note that $\sA \otimes_\QQ \CC = \CC[G]$)
	\[
	\Hom_{\CC[G]}(M,N) = 0 \iff \Hom_{Z(\CC[G])}(M,N) = 0.
	\]
The right  hand side $\Hom_{Z(\sA \otimes_\QQ \CC)} (M, N) $   clearly contains the left hand side.

By  semisimplicity, 
  each  representation $M$ splits as a direct sum of irreducible representations,
$$  M =  \sum_{\chi \in Irr}  M_{\chi }, M_{\chi } = W_{\chi } \otimes_{\CC[G]} (\CC^r),$$
where $ \CC^r$ is a trivial representation of $G$.

By bilinearity  we may assume that $M = W_\chi$ and $N = W_{\chi'}$ are simple modules associated to irreducible characters $\chi, \chi'$ of $G$.

		Then, by the Lemma of Schur, the left hand side
		 $\Hom_{\sA \otimes_\QQ \CC} (M, N)$  is $=0$ for $\chi' \neq \chi$,
and  isomorphic to $\CC$  for $\chi'  =  \chi$.
 
For the right hand side, it suffices  to prove that $\Hom_{Z(\sA \otimes_\QQ \CC)} (M, N) = 0$
for  $\chi' \neq \chi$, when  $M = W_{\chi }, N = W_{\chi' }$.

However, $e_\chi$ acts as the identity on $M$ and as zero on $N$,
hence $\psi \in \Hom_{Z(\sA \otimes_\QQ \CC)} (M, N) $
implies
$$ \psi (v) =   \psi  ( e_{\chi} v) =  e_{\chi} ( \psi (  v))  = 0,$$
as we wanted to show.

		This shows the statement.

\end{proof}

We have more generally:

\begin{prop} Let $\sA$ be a semisimple $\QQ$-algebra of finite dimension, and let 
$(V, \sA,  r)$ be an action on a rational Hodge structure $V$, Then $r$ is   rigid   if and only if $(V, Z(\sA),   r')$ is  rigid;
here $Z(\sA)$ is the centre of $\sA$ and $r'$ is the restriction of $r$.
\end{prop}

\Proof
More generally, if $M, N$ are $\sA \otimes \CC$-modules, then we claim that
$$  \Hom_{\sA \otimes \CC} (M, N) = 0 \Leftrightarrow \Hom_{Z(\sA \otimes \CC)} (M, N) = 0.$$

By bilinearity of both sides, and by semisimplicity (each module splits as a direct sum of irreducibles) we can assume that $M,N$ are simple
modules and that $\sA$ is a simple algebra.

By Schur's Lemma the left hand side is non zero exactly when  $M$ and $N$ are   isomorphic. The left hand side is contained in  the right hand side, so it suffices to show that the right hand side
is nonzero exactly when  $M$ and $N$ are   isomorphic. But (\cite{jacobson2}, Lemma 1,  page 205) any two irreducible modules over a simple Artininian ring are isomorphic.

\qed

 \begin{rem} \label{idempotents}
		We have $\CC[G] = \bigoplus_\chi \CC[G] \cdot e_\chi$. 
		
		 Working instead over a  field $K  $ of characteristic $0$, an algebraic extension of $\QQ$ (so $ \QQ \subset K \subset \CC$), the decomposition of $K[G]$ into simple summands is (see \cite{Yamada}, Proposition 1.1)
		 again provided by central idempotents in $K[G]$,
		\[
		K[G] = \bigoplus_{[\chi]} K[G]e_{K}(\chi), \; \; \;  e_{K}(\chi) := \sum_{\chi^\sigma \in [\chi]} e_{\chi^\sigma},
		\]
		where the first sum runs over the set of $\Ga$-orbits $[\chi]$  in the set   all irreducible characters $\chi$ of $G$; here  $\Ga$ is 
		 the Galois group $\Gal(K(\chi)/K)$ of the field  extension $K(\chi)$ of $K$,   generated by 
		 the values of all the characters $\chi$, i.e., by  $\{\chi(g) \, | \, g \in G, \chi \in \Irr (G) \}$.
		 
		 And the centre of $K[G]$ is a direct sum of fields
		 $$Z(K[G]) =   \bigoplus_{[\chi]} F_{[\chi]},$$
		  where the field $F_{[\chi]}$ is the centre  (for the last isomorphism, see \cite{Yamada}, Proposition 1.4)  $$F_{[\chi]}: = Z(K[G]) e_{K}(\chi) \cong K(\{\chi(g)| g \in G \}) $$
		  of the algebra $K[G]e_{K}(\chi)$, and enjoys the property that  $F_{[\chi]} \otimes_K \CC =  \bigoplus _{\chi \in [\chi]} \CC e_{\chi^\sigma}$.

	\end{rem}
	
	The next lemma explains the relation occurring between finite groups and CM-fields.
	
	\begin{lemma}\label{cm}
	The centre of the group algebra $Z(\QQ[G])$ splits as a direct sum of number fields, $Z(\QQ[G]) = F_1 \oplus ... \oplus F_l$
	which are either  totally real,  or  CM-fields.
	\end{lemma}
	\Proof
	Write  $m := |G|$,  let $\zeta_m$ be  a primitive $m$-th root of unity and let $d $ be the number of conjugacy classes in $G$,
	which equals the number of irreducible representations of $G$. Then 
			\[
			F_j \subset Z(\QQ[G]) \subset Z(\QQ(\zeta_m)[G]) \cong_{\QQ-\text{alg.}} \QQ(\zeta_m)^d,
			\]
			where we used in the last isomorphism that every complex representation of $G$ is defined over $\QQ(\zeta_m)$. Hence $F_j$ embeds into the cyclotomic field $\QQ(\zeta_m)$. The extension $\QQ(\zeta_m)/\QQ$ is Galois with group $\Gal(\QQ(\zeta_m)/\QQ) \cong (\ZZ/m\ZZ)^*$ (the isomorphism maps $\fie_a \in \Gal(\QQ(\zeta_m)/\QQ)$, 
			such that $\fie_a(\zeta_m) = \zeta_m^a$,  to $a \in (\ZZ/m\ZZ)^*$), so by the Main Theorem of Galois Theory, there is a subgroup $H$ of $\Gal(\QQ(\zeta_m)/\QQ)$, such that $F_j \cong \QQ(\zeta_m)^H$ (the subfield of $\QQ(\zeta_m)$ fixed by the action of $H$). If $-1 \in H$ (which corresponds to $\fie_{-1}$, the complex conjugation), the field $F_j$ is totally real, otherwise $F_j$ is a CM-field.\par\hfill \par
	
	\qed
	
\section{Proof of Theorem \ref{ekedahl}}

Fix  now  an action $(V,\sA,r)$ and assume that
\begin{align} \label{assump}
\sA \text{ is commutative.}
\end{align}

Since $\sA$ is commutative, $\sA$ is a direct sum of number fields, 
	\[
	\sA = F_1 \oplus ... \oplus F_l.
	\]

Assume that we have a homomorphism of algebras $\s : \sA \ra \CC$. For each idempotent $e$ of $\sA$,  $\s (e)$ is an idempotent of $\CC$, hence 
$\s (e) = 1$ or $\s (e) = 0$.  In $\sA$, the identity element $1$  is a sum of idempotents 
$$ 1 = 1_{F_1} + \dots +  1_{F_l},$$
and if $\s \neq 0$, then $\s (1) = 1$. This implies that for such a homomorphism $\s$ there is exactly one $j \in \{1, \dots, l\}$, such that $\s ( 1_{F_j}) = 1$,
and, for $i\neq j$, we have $\s ( 1_{F_i}) = 0$.

 Let then $\sC = \{\sigma_1, ..., \sigma_k \}$ be the set  of all the distinct $\QQ$-algebra homomorphisms $\sA \to \CC$: then these 
 homomorphisms $\s_j : \sA \to \CC$  are obtained  as the composition of one of the projections $\sA \to F_h$ with an embedding $F_h \hookrightarrow \CC$\label{page} 
(hence  $k = \sum_h [F_h:\QQ] = \dim_\QQ \sA$).

Define now (as in Example \ref{imaginary})  the $\sA$-module $\CC_{\sigma_j}$ as the vector space $\CC$
endowed with the action of $\sA$ such that $$x \cdot z := \sigma_j(x) \cdot z.$$ 

Hence we have a splitting of $\sA$-modules 

$$\sA \otimes_\QQ \CC = \bigoplus_{j=1}^l (F_j \otimes_\QQ \CC) = \bigoplus_{j=1}^k \CC_{\sigma_j}.$$

We now show that 
we have a splitting in the category of rational Hodge structures
 $$V  = V_1 \oplus ... \oplus V_l,$$
where $V_i$ is an $F_i$-module, and an $\sA$-module via the surjection $\sA \to F_i$.

We simply 
 define $V_j : =  1_{F_j} \cdot V$. We have a splitting of modules 
$$V  = V_1 \oplus ... \oplus V_l,$$
since for $i\neq j$,  $1_{F_i} 1_{F_j}=0$, and $$ v = 1 \cdot v = (1_{F_1} + \dots + 1_{F_l } ) v = : v_1 + \dots + v_l. $$

It is a splitting in the category of rational Hodge structures because
each element of $\sA$ preserves the Hodge decomposition, hence $V_j $ is a sub-Hodge structure of $V$.

Therefore the action $r$ is a direct sum of actions

$$r_j : F_j \ra \End_{\sH^1} (V_j)$$

Each $r_j$  induces, by tensor product,  a homomorphism of rings 
\begin{align*}
F_j  \otimes_\QQ \CC  \to \End(V_j \otimes_\QQ \CC) = \End(V_j^{1,0} \oplus V_j^{0,1}),
	\end{align*}
and a splitting of $\sA$-modules
	\begin{align*}
	 V \otimes \CC =   V^{1,0} \oplus V^{0,1}   =   \bigoplus_{j=1}^k (V^{1,0}_{\sigma_j} \oplus V^{0,1}_{\sigma_j})
	\end{align*} 
	 where $V_{\sigma_j}$ is the character subspace on which $\sA$ acts via $x \cdot v := \sigma_j(x) \cdot v.$
This holds for the following reason: each $V_j$ is  an $F_j$ module; and since $ F_j$ is a number field, 
then $F_j = \QQ[x]/ P(x)$, where $P$ is irreducible, and $r_j(x) $ is an endomorphism $a_j$ of $V_j$ with minimal polynomial $P$ (a polynomial with distinct roots).
In particular, $a_j$ is diagonalizable over  $V_j \otimes_\QQ \CC$,
and each diagonal entry yields some embedding $\s_h$ of $F_j$ into $\CC$.

\begin{rem}
The rigidity of $(V, \sA, r)$ is equivalent to the fact  that for each $\sigma_j \in \sC$ either  $V^{1,0}_{\sigma_j}$ or $V^{0,1}_{\sigma_j}$
is zero,
  in particular, since  $\overline{V^{1,0}_{\sigma_j}} = V^{0,1}_{\overline{\sigma_j}}$,
  no real $\sigma_j$ appears either  in $V^{1,0}$ or in $V^{0,1}$.

\end{rem}

Following a terminology similar to   the one introduced in \cite{topmethods}, we define the notion of Hodge-type.

\begin{defin}
Define the \textit{Hodge-type} of an   action of $\sA$ by the function 
$\tau_V : \sC  \ra \NN$,   such that
 $$ \tau_V (\s) : = dim_{\CC} \ V^{1,0}_{\sigma}.$$

  Hodge symmetry translates into 
 $$ (HS) \ \tau_V(\s) +  \tau_V(\bar{\s}) = \dim_{\CC} V_{\sigma},$$ 
which  implies in particular that if  we have a real embedding, i.e.  $\s = \overline {\s}$,
then $\tau_V(\s) =  \frac{1}{2} \dim_{\CC} V_{\s}$.

Moreover,   if Hodge symmetry holds, 
the action is rigid if and only if  $$(R) \ \tau_V (\s)  \cdot   \tau_V(\bar{\s} )= 0, \ \ \forall \s .$$
\end{defin}

\begin{prop} If $(V, \sA,  r)$ is  rigid, then it is determined by the  $\sA$-module $V$ and by the Hodge-type.

Conversely, if $V$ is an $\sA$-module, and there is a Hodge structure such that 
$$ (HS) \ \tau_V(j) +  \tau_V(\bar{j}) = \dim_{\CC} V_{\sigma_j},$$ whenever  $\s_{\bar{j}} = \overline {\s_j}$, 
and moreover $$(R) \ \tau_V (j)  \cdot   \tau_V(\bar{j} )= 0 \ \ \forall j ,$$ then this Hodge structure determines   a rigid action $(V, \sA,  r)$.
\end{prop}

\Proof
In one direction, the Hodge-type determines $V^{0,1} , V^{1,0}$, since, $\sA$ being commutative, $V$ splits into  
 character spaces $V_{\sigma_j}$,  and the function $\tau_V$ determines  whether $V_{\sigma_j} \subset V^{0,1} $, or $V_{\sigma_j} \subset V^{1,0} $.
 
 In the other direction, the given Hodge structure is preserved by the action of $\sA$ hence we have an action in the category of rational Hodge structures.

\qed

	\begin{lemma} \label{special-cases}  Assume that we have a rigid action $(V,\sA,r)$ of split type, where  $$\sA = F_1 \oplus ... \oplus F_l$$ is commutative
	and  each  $F_i$ is a field.
		\begin{itemize}
	\item[i)] If $l=1$ (so $\sA=: F$ is a field), $V \cong W^n$ in $\sH^1$, where $W$ is a Hodge structure on $F$.
	\item[ii)]  the  rational Hodge structure $V$ splits as a direct sum 
	$$ V = W_1^{n_i} \oplus ... \oplus W_l^{n_l},$$ where $W_j$ is a Hodge structure on $F_j$ and $n_j \geq 0$. 
	
			\end{itemize}
	\end{lemma}
	
	\begin{proof}
Assertion  i):  here $V$ is an $F$-vector space, and so $f \colon V \stackrel{\sim}{\to} F^n$ as vector spaces. As we observed the rigidity of $(V,F,r)$  implies that all embeddings of $F$ into $\CC$ appear in either $V^{1,0}$ or $V^{0,1}$, hence  $F$ has no real ones.
  Let $\sigma_1, ..., \sigma_d$ be the embeddings of $F$ appearing in $V^{1,0}$, so that $\overline{\sigma_1}, ..., \overline{\sigma_d}$ are the ones appearing in $V^{0,1}$. Define a Hodge structure $W$ on $F$ according to the type of $V$, i.e. as follows:
	\[
	W \otimes_\QQ \CC = W^{1,0} \oplus W^{0,1}, \text{ where } W^{1,0} = \bigoplus_{j=1}^d \CC_{\sigma_j}, \; \; W^{0,1} = \bigoplus_{j=1}^d \CC_{\overline{\sigma_j}},
	\]
	Then $f_\CC \colon V \otimes_\QQ \CC \to (W \otimes_\QQ \CC)^n$ is an isomorphism of $\CC$-vector spaces together with an $F$-action. 
	
   Assertion   ii) follows immediately from assertion i),   since we have the splittings $\sA = F_1 \oplus ... \oplus F_l$ and 
	$ V =  V_1 \oplus  \dots  \oplus V_l$, and the $\sA$-rigidity 	of $V$ implies the $F_j$-rigidity of $V_j$ for all $j=1, \dots , l$, hence we can apply step i) to each $V_j$.
	
		\end{proof}

	 The crucial Proposition from which the  proof of Theorem \ref{ekedahl} follows is now

	 \begin{prop} \label{polarizable}
		 If $(V,\QQ[G],r)$ is rigid, then  $V$   polarizable. 
		\end{prop}
		
		\begin{proof}
			 
	First of all, if $(V,\QQ[G],r)$ is rigid, then $(V,Z(\QQ[G]),r)$ is rigid by Proposition \ref{reduce_to_comm}.	
			 The assumption that $(V,Z(\QQ[G]),r)$ is rigid implies now that if  some field $F_j$ does not act as $0$ on $V$, 
			 then $F_j$ is necessarily  a CM-field by Lemma \ref{cm} and the previous remarks. By Lemma \ref{special-cases}, 
			  the  rational Hodge structure $V$ splits as a direct sum $W_1^{n_i} \oplus ... \oplus W_l^{n_l}$, where $W_j$ is a Hodge structure on $F_j$ and $n_j \geq 0$. \\
			 To give a polarization on $V$, it therefore suffices to show the existence of a polarization for a Hodge structure $W_j$
			on a CM-field  $F_j$.  But this was shown in Example \ref{CM}.

		\end{proof}

Ekedahl's Theorem is therefore proven. 

\smallskip

\section{Final remarks}

Assume that $X: = T$ is a complex torus of dimension $\geq 3$, and that $Y = T/G$ has only isolated singularities. 

Schlessinger   showed in \cite[Theorem 3]{Schlessinger} that every deformation of the analytic  germ of $Y$ at each  singular point of $Y$ is  trivial. 

Hence for every deformation $ \sY \ra B$   of $Y$  (we write informally $\sY$ as $ \{ Y_t\}_{t \in B}$) $Y_t$
has the same singularities as $Y$, and in particular it follows easily
that  $Y_t \setminus \Sing(Y_t)$ and $Y \setminus \Sing(Y)$ are diffeomorphic and a fortiori one
has an isomorphism   $\pi_1(Y_t \setminus Sing(Y_t)) \cong \pi_1(Y \setminus Sing(Y)) \cong \pi_1(\sY \setminus Sing(\sY ))$. Therefore the
surjection  $\pi_1(Y \setminus Sing(Y)) \to G$ induces a surjection $\pi_1(\sY \setminus Sing(\sY ) ) \to G$.

Whence, by Grauert's and Remmert's extension of Riemann's Existence Theorem, cf. \cite[Satz 32]{Grauert-Remmert}, $Y_t$ 
and $\sY$ have  respective  Galois covers $X_t$ and $\sX$ with group $G$. Hence, the action of $G$ extends to
the family $\sX$, and each deformation of $Y$ yields a deformation of the pair $(T,G)$.

The conclusion is that $Y$ is rigid if and only if the action of $G$ on $T$ is rigid. 
On the other hand, Ekedahl's theorem implies then  that if $Y$ is rigid, then $Y$ is projective.

Therefore in this case one cannot get a counterexample to the Kodaira property via rigidity.
We show  more generally  in the appendix  that any such a quotient $Y = T/G$ with  only isolated singularities satisfies the Kodaira property,
 since any action can be approximated by a projective one.

 An interesting question is: in the case where $Y$ is rigid, is it true that a minimal resolution of $Y$ is also rigid?

\bigskip 

\thanks{\textbf{Acknowledgements:} The authors would like to thank professor Yujiro Kawamata for
 a useful conversation and a  referee for careful reading of the  manuscript,   and   for pointing out that  for such rigid actions on tori
 we have $H^{2,0}(T)^G = 0$.}

\bigskip 

\section{Appendix by Fabrizio Catanese, Andreas Demleitner and Beno\^it Claudon}

 Ekedahl's theorem has the advantage of elucidating the structure of (rigid and non rigid)  actions  of a finite group $G$ on a complex torus.
 
 The method of period mappings,  used by  Green and Voisin (see proposition 17.20 and Lemma 17.21 of \cite{voisinbook}) for showing the density of algebraic tori (non constructive, since it uses  the implicit functions theorem),
 was  used by Graf in \cite{graf} to obtain a  general criterion, from which follows    the following theorem.
  
 \begin{theo}
 Let $(T,G)$ be a group action on a complex torus. Then there are arbitrarily small deformations $(T_t,G)$ of the action where $T_t$ is projective.  
 \end {theo}

 \Proof
 Given a complex torus
 $$T: = (\Lam \otimes_{\ZZ} \CC )/ (\Lam \oplus V^{1,0} ),$$
 set, as in section 2,
 \[
V \otimes_{\QQ} \CC = U \oplus \overline{U} = : V^{1,0} \oplus V^{0,1}.
\]
The  Teichm\"uller space of $T$ is  an open set $\sT$ in the Grassmann variety $Gr (n, V \otimes_{\QQ} \CC)$,
$$ \sT = \{ U_t  | U_t \oplus \overline{U_t} = V \otimes_{\QQ} \CC \},$$
parametrizing Hodge structures. By abuse of notation we shall use the notation $ t \in \sT$ for the points of Teichm\"uller space.
 
The deformations of the pair $(T,G)$ are parametrized by the submanifold $\sT^G$ of the  fixed points  for the action of $G$,
 which correspond to the set of the subspaces $U_t$ which are $G$-invariant. 
 
 The tangent space to $\sT^G$ at the point $(T,G)$ is, as seen in section 2,  the subspace 
 $$H^1(\Theta_{\torus})^G \subset  H^1(\Theta_{\torus}) =  H^1(\Oh_{\torus}) \otimes_\CC H^0(\Om^1_{\torus})^\vee = \overline{U}^\vee \otimes_\CC U.$$
 
 Over $\sT^G$ we have the Hodge bundle 
 $$ F^1 \subset \sT^G \times \wedge^2  (V \otimes_{\QQ} \CC)^{\vee} \   \ s.t. \ \  F^1_t = H^{1,1} (T_t) \oplus  H^{2,0} (T_t).$$
 
 Since  the  family of complex tori  is differentiably trivial  there is a canonical isomorphism 
 $$  \wedge^2  (V \otimes_{\QQ} \CC)^{\vee} = H^2 (T , \CC) \cong H^2 (T_t , \CC).$$
 
 This allows to define a holomorphic mapping $\psi : F^1  \ra H^2 (T , \CC)$ induced by the second projection.
 
 We can indeed consider the subbundle (defined over $\sT^G$)
 $$ (F^1)^G  \subset \sT^G \times H^2 (T , \CC)^G  \   \ s.t. \ \  (F^1)^G_t = H^{1,1} (T_t)^G \oplus  H^{2,0} (T_t)^G,$$
 and the corresponding holomorphic mapping $\phi : (F^1)^G  \ra H^2 (T , \CC)^G$ induced by the second projection.
\medskip

 {\bf Step 1:} Let $\eta$ be a K\"ahler metric on $T$. By averaging, we replace $\eta$ by $\sum_g g^* (\eta)$ and we can assume that $\eta$ is $G$-invariant.
 
 Let $\omega \in H^{1,1} (T) \cap H^2 (T_t , \RR)^G$ be the corresponding K\"ahler class.
 \medskip
 
 {\bf Step 2:} Setting $T = : T_0$, the map $\phi$  is a submersion at the point $(0, \omega)$.
 
 \medskip
 
 Before proving step 2, let us see how  the theorem follows.
 
 Let  $\sD$ be a sufficiently small neighbourhood of $\omega$ inside
 $$H^2 (T , \CC)^G = H^2 (T , \QQ)^G \otimes_{\QQ} \CC.$$
 
  For each class $\xi \in H^2 (T , \QQ)^G \cap \sD$, there is therefore a $(t,\xi)$ in a small neighbourhood $\sD'$ of $(0, \omega)$  such that 
  $$ \xi \in (F^1)^G_t = H^{1,1} (T_t)^G \oplus  H^{2,0} (T_t)^G.$$
  Since $\xi$  is real, $ \xi \in  H^{1,1} (T_t)^G \cap H^2 (T , \QQ)^G $. Taking $\sD$ sufficiently small, 
  the class $\xi$ is also positive definite,  hence $\xi$ is the class of a polarization on $T_t$.
  
 Shrinking   $\sD$ and $\sD'$, we obtain that $t \in \sT^G$ tends to $0$ (the point corresponding to the torus $T$).
 Hence the assertion of the theorem is proven.\\

{\bf Proof of Step 2.}

The tangent space to  $ (F^1)^G$ at the point $(0, \omega)$ is the direct sum 
$$ H^1(\Theta_{\torus})^G \oplus  (F^1)^G_0 = H^1(\Theta_{\torus})^G \oplus H^{1,1} (T)^G \oplus  H^{2,0} (T)^G, $$
and the derivative of $\phi$ is the direct sum of $\cup \omega, \iota$,
where $\iota$ is the inclusion $(F^1)^G_0 \subset   H^{2} (T, \CC)^G, $
while the cup product with $\omega \in $ yields a linear map 
$$ \beta : H^1(\Theta_{\torus})^G   \ra H^2 (T, \hol_T)^G = H^{0,2} (T)^G \subset H^{2} (T, \CC)^G.$$ 
 
 Whence $\phi$ is a submersion at $(0, \omega)$ iff $\beta$ is surjective.
 
 Now, $\beta$ is surjective if  the cup product with $\omega$ yields a surjection 
 $$ \beta' :  H^1(\Theta_{\torus})   \ra H^2 (T, \hol_T)$$
 (taking the subspace of  $G$-invariants is an exact functor).
 
 Observe that $H^2 (T, \hol_T) = \wedge^2 (\overline{U}^\vee)$, while 
 $$H^{1,1} (T) = H^1 (\Omega^1_T) = \overline{U}^\vee \otimes_{\CC} U^{\vee}.$$
 
 Cup product with $\omega$ is the composition of two linear maps 
 $$ H^1(\Theta_{\torus})  \ra H^2 ( \Theta_{\torus} \otimes_{\hol_T} \Omega^1_T) \ra  H^2 (T, \hol_T),$$ 
 where the second map is induced by contraction.
 
 It can be also seen as the composition of three linear maps: 
 $$ H^1(\Theta_{\torus})= \overline{U}^\vee \otimes_\CC U  \ra  (\overline{U}^\vee \otimes_\CC U) \otimes_\CC (\overline{U}^\vee \otimes_{\CC} U^{\vee})
 \ra \overline{U}^\vee \otimes_\CC \overline{U}^\vee  \ra \wedge^2 (\overline{U}^\vee) = 
   H^2 (T, \hol_T). $$ 
 
 Since the last linear map is a surjection, it suffices to show that the composition of the first two maps yields a surjection
 $$b : \overline{U}^\vee \otimes_\CC U  \ra \overline{U}^\vee \otimes_\CC \overline{U}^\vee .$$

 Since $\omega$ is a K\"ahler class, there exists a basis $u_i$ of $U$ such that 
 $$ \omega = \sum_i \overline{u_i^\vee} \otimes_\CC u_i^\vee  .$$
 
 Hence $$ \sum_{h,k} a_{h,k} \overline{u_h^\vee} \otimes_\CC u_k \ra  \sum_{h,k} a_{h,k} \overline{u_h^\vee} \otimes_\CC \overline{u_k^\vee}$$ 
  and $b$ is an isomorphism.
  
 \qed

\end{document}